\documentclass[12pt]{amsart}
\usepackage{a4wide}
\usepackage{graphicx}
\usepackage{color}
\usepackage{amsmath}
\usepackage{enumitem}
\allowdisplaybreaks

\let\pa\partial  
\let\na\nabla  
\let\eps\varepsilon  
\newcommand{\N}{{\mathbb N}}  
\newcommand{\R}{{\mathbb R}}

\newtheorem{theorem}{Theorem}   
\newtheorem{lemma}[theorem]{Lemma}   
   
\newtheorem{remark}[theorem]{Remark}

\begin{document}  

\title[Cross-diffusion systems and fast-reaction limits]{Cross-diffusion systems
and fast-reaction limits}

\author[E. Daus]{Esther S. Daus}
\address{Institute for Analysis and Scientific Computing, Vienna University of  
	Technology, Wiedner Hauptstra\ss e 8--10, 1040 Wien, Austria}
\email{esther.daus@gmail.com} 

\author[L. Desvillettes]{Laurent Desvillettes}
\address{Universit\'{e} Paris Diderot, Sorbonne Paris Cit\'{e}, Institut de Math\'{e}matiques de Jussieu-Paris Rive Gauche, UMR 7586, CNRS, Sorbonne Universit\'{e}s, UPMC Univ. Paris 06, F-75013, Paris, France}
\email{desvillettes@math.univ-paris-diderot.fr}

\author[A. J\"ungel]{Ansgar J\"ungel}
\address{Institute for Analysis and Scientific Computing, Vienna University of  
	Technology, Wiedner Hauptstra\ss e 8--10, 1040 Wien, Austria}
\email{juengel@tuwien.ac.at} 

\date{\today}

\thanks{The authors have been partially supported by the Austrian-French
project ``Amad\'ee'' of the Austrian Exchange Service (\"OAD).
The first and second authors acknowledge partial support from
the French ``ANR blanche'' project Kibord, grant ANR-13-BS01-0004, and 
from the Universit\'e Sorbonne Paris Cit\'e, in the framework of the 
``Investissements d'Avenir'', grant ANR-11-IDEX-0005.
The first and last authors acknowledge partial support from   
the Austrian Science Fund (FWF), grants P27352, P30000, F65, and W1245} 

\begin{abstract}
The rigorous asymptotics from reaction-cross-diffusion systems for three 
spe\-cies with known entropy to cross-diffusion systems for two variables
is investigated. The equations are studied in a bounded domain with
no-flux boundary conditions. The global existence of very weak
(integrable) solutions and the rigorous fast-reaction limit are proved. 
The limiting system inherits
the entropy structure with an entropy that is not the sum
of the entropies of the components. Uniform estimates are derived from
the entropy inequality and the duality method.
\end{abstract}

 

\maketitle


\section{Introduction}

The analysis of cross-diffusion systems with unknowns $u_1,\ldots,u_n$
often relies on the existence of a convex Lyapunov functional, 
called here an entropy, which provides suitable
gradient estimates \cite{ChJu04,DLM14,DLMT15,Jue15}. Given the partial 
differential system, the difficulty is
to identify such an entropy functional. Often, it is of the
form $\int_\Omega\sum_{i=1}^n h_i(u_i)dx$ for convex functions $h_i$, 
which only depend on $u_i$;
see the examples in the aforementioned references.
In this paper, we identify entropy functionals for certain cross-diffusion systems
that are generally not the sum of all $h_i(u_i)$.

Our approach is to consider first reaction-cross-diffusion systems for which
a Lyapunov functional is known to exist and which is of the form
$\int_\Omega\sum_{i=1}^n h_i(u_i)dx$. Then we perform the limit of vanishing
relaxation times that are related to the reaction terms. The limiting system
consists of cross-diffusion equations, which possesses an entropy inherited from
the original system and where the variables $u_i$ are related by an algebraic
relation coming from the reaction terms. This strategy enlarges the class
of cross-diffusion systems with an entropy structure by providing examples
for which the entropy cannot be easily found in another way.

As an example of this approach, we consider reaction-cross-diffusion equations
whose reaction terms correspond to one reversible reaction of the form
$A \rightleftharpoons B + C$. More specifically, we study the equations
\begin{equation}\label{1.eq}
  \pa_t u_i^\eps - \Delta F_i(u^\eps) = Q_i(u^\eps) \quad\mbox{in }
	\Omega,\ t>0,\ i=1,2,3,
\end{equation}
supplemented with no-flux boundary and initial conditions,
\begin{equation}\label{1.bic}
  \na F_i(u^\eps)\cdot\nu = 0\quad\mbox{on }\pa\Omega, \quad
	u_i^\eps(0) = u_i^I\quad\mbox{in }\Omega,\ i=1,2,3.
\end{equation}
Here, $u^\eps=(u_1^\eps,u_2^\eps,u_3^\eps)$, 
$\Omega\subset\R^d$ ($d\ge 1$) is a bounded domain with smooth boundary,
and $\nu$ is the exterior unit normal vector to $\pa\Omega$.
The unknowns $u_i^\eps$ can be interpreted as chemical concentrations, but 
generally they are just densities in some diffusive system whose application is 
not specified. The nonlinear functions contain cross-diffusion terms,
\begin{equation}\label{1.F}
\begin{aligned}
  F_1(u^\eps) &= f_1(u_1^\eps) + f_{12}(u_1^\eps,u_2^\eps), \\
	F_2(u^\eps) &= f_2(u_2^\eps) + f_{21}(u_1^\eps,u_2^\eps), \\
	F_3(u^\eps) &= f_3(u_3^\eps),
\end{aligned}
\end{equation}
and the reaction terms are given by 
\begin{equation}\label{1.Q}
\begin{aligned}
	Q_1(u^\eps) &= -\eps^{-1}\big(q_1(u_1^\eps) - q_2(u_2^\eps)q_3(u_3^\eps)\big), \\
	Q_2(u^\eps) &= \eps^{-1}\big(q_1(u_1^\eps) - q_2(u_2^\eps)q_3(u_3^\eps)\big), \\
  Q_3(u^\eps) &= \eps^{-1}\big(q_1(u_1^\eps) - q_2(u_2^\eps)q_3(u_3^\eps)\big).
\end{aligned}
\end{equation}
The constraints on functions $f_i$, $f_{ij}$, and $q_i$ are specied in Assumptions (A1)-(A5) below.
The parameter $\eps>0$ models the inverse of a reaction rate or, generally, 
a relaxation time.

Without diffusion terms, the corresponding system of ordinary differential equations
is known to possess the Lyapunov functional
\begin{equation}\label{1.h}
  h(u) := \sum_{i=1}^3\int_1^{u_i}\log q_i(s)ds.
\end{equation}
When $q_i(u_i)=u_i$, we recover the physical entropy for the reaction 
$A \rightleftharpoons B + C$, i.e. $h(u)=\sum_{i=1}^3 u_i(\log u_i-1)$.
The functional $\int_\Omega h(u)dx$ is still a Lyapunov functional 
if the diffusion terms are
given by $\Delta f_i(u_i)$. In this paper, we allow for the cross-diffusion
terms $f_{12}(u_1,u_2)$ and $f_{21}(u_1,u_2)$. Clearly, an additional assumption
is then needed to guarantee that \eqref{1.h} is still an entropy for \eqref{1.eq}. We show that this is the case under a ``weak cross-diffusion'' condition;
see Assumption (A5) below.

The fast-reaction limit $\eps\to 0$ in \eqref{1.eq} leads formally to
the system
\begin{align*}
  & \pa_t(u_1+u_2) = \Delta(F_1(u)+F_2(u)), \\
	& \pa_t(u_1+u_3) = \Delta(F_1(u)+F_3(u)), \\
  & q_1(u_1) = q_2(u_2)q_3(u_3)\quad\mbox{in }\Omega,\ t>0.
\end{align*}
Under certain conditions on $q_i$, this system can be formulated in terms
of the variables $v=u_1+u_2$ and $w=u_1+u_3$, leading to
\begin{equation}\label{1.vw}
  \pa_t v = \Delta G_1(v,w), \quad \pa_t w = \Delta G_2(v,w),
\end{equation}
where 
\begin{align*}
  & G_1(v,w) = (F_1+F_2)(u_1,u_2,u_3), \quad G_2(v,w) = (F_1+F_3)(u_1,u_2,u_3), 
	\quad\mbox{and} \\
	& u_1 = q_1^{-1}\big(q_2(u_2(v,w))q_3(u_3(v,w))\big), \quad 
	u_2 = u_2(v,w),\quad u_3 = u_3(v,w).
\end{align*}
Formally, the limit entropy
\begin{align*}
  \int_\Omega h_0(v,w) dx
	&= \int_\Omega\bigg\{\int_1^{q_1^{-1}(q_2(u_2(v,w))q_3(u_3(v,w)))}
	\log q_1(s)ds \\
	&\phantom{xxxxx}{}+ \int_1^{u_2(v,w)}\log q_2(s)ds 
	+ \int_1^{u_3(v,w)}\log q_3(s)ds\bigg\}dx
\end{align*}
is a Lyapunov functional for \eqref{1.vw}.
A simple example is given in Remark \ref{rem.ex}.

Fast-reaction limits in reaction-diffusion equations have been studied since
about 20 years. These limits are of importance in mass-action kinetics chemistry
to reduce a system of many components to a (nonlinear) system with less
equations. One of the first papers is \cite{HHP96}, where a fast-reaction limit in a
system consisting of one parabolic and one ordinary differential equation was
performed. Later, the fast-reaction limit in a two-species diffusion system
was shown, leading to a nonlinear diffusion equation \cite{BoHi03}. 
Fast irreversible reactions for two species, studied in \cite{BoPi12}, led to
a Stefan-type limit problem with a moving interface, which represents the 
chemical reaction front. Systems for three species with Lotka-Volterra-type
interactions \cite{MuNi11} or with reversible reactions
\cite{BPR12} were also analyzed.
A unified approach for self-similar fast-reaction limits was given in \cite{CrHi16}.

In \cite{HeTa16}, the fast-reaction limit in a system containing a parabolic 
equation on the domain boundary (volume surface diffusion model) was proved.
Here, the limit problem is the heat equation with a dynamic boundary condition.
A combination of the fast-reaction limit and homogenization techniques
has given a two-scale reaction-diffusion system with a moving boundary 
traveling within the microstructure \cite{MeMu10}. Finally, 
asymptotic limits related to fast reactions were investigated in 
reaction-diffusion equations from population dynamics \cite{CoDe14,IMN06,Mur12}. 
Here, the small parameter describes an averaged time within which two types of 
species convert to each other. If the conversion is of nonlinear type, 
the limit problem becomes a cross-diffusion system.

A three-species system with power-like reaction functions $q_i$ was
investigated in \cite{BPR12,BoRo15}, proving the existence of mild
solutions employing a semigroup approach \cite{BoRo15} and the fast-reaction limit
using entropy and duality techniques \cite{BPR12}. 

The main difference between
our approach and the results of \cite{BPR12,BoRo15}, and the main novelty of this paper,
 is that we allow for
cross-diffusion terms in the original reaction-diffusion system, at least
``weak cross-diffusion'' as specified in Assumption (A5). 
Interestingly, the Lyapunov functional structure is still kept when adding
cross-diffusion to a certain extent. This leads to a much larger set of
cross-diffusion systems than known up to now, for which a Lyapunov functional
can be produced.

Before we detail our main results, we need some assumptions.
First, we introduce the notation 
$$
  \R_+=[0,\infty), \quad \R_+^*=(0,\infty), \quad Q_T=\Omega\times(0,T).
$$
The functions $F_i$ and $Q_i$ are extended continuously to $\R^3$ by setting
$F_i(u)=F_i(|u_1|,|u_2|,$ $|u_3|)$ and $Q_i(u)=Q_i(|u_1|,|u_2|,|u_3|)$
for any $u=(u_1,u_2,u_3)\in\R^3$. Finally, we set $|u|^2=\sum_{i=1}^3 u_i^2$ for
$u=(u_1,u_2,u_3)\in\R^3$.

\bigskip
\noindent{\bf Assumptions.} We impose the following conditions.
\renewcommand{\labelenumi}{(A\theenumi)}
\begin{enumerate}[leftmargin=10mm]
\item Nonlinear diffusion: $f_i\in C^1(\R_+;\R_+)$ satisfies
$f_i(s)=g_i(s)s$ for $s\ge 0$, where $g_i\in C^0(\R_+;\R_+)\cap C^1(\R_+^*;\R_+^*)$,
and $f'_i(s)\ge \kappa_1>0$ for all $s\ge 0$ and for some $\kappa_1>0$.

\item Cross-diffusion: $f_{ij}\in C^0(\R_+^2;\R_+)$ satisfies
$f_{ij}(s_1,s_2)=g_{ij}(s_1,s_2)s_i$, where
$g_{ij}\in C^0(\R_+^2;\R_+)\cap C^1((\R^*_+)^2;\R_+^*)$, and
$\pa_1 f_{ij}(s_1,s_2)\ge 0$, $\pa_2 f_{ij}(s_1,s_2)\ge 0$ for all $s_1$, $s_2\ge 0$,
$i,j=1,2$, $i\neq j$.

\item Reaction terms I: $q_i\in C^1(\R_+)$ satisfies $q_i(0)=0$,
$q_i'(s)>0$ for all $s>0$, and $q_i(s_0) \ge 1$ for some $s_0>0$.

\item Reaction terms II: 
There exist $C_q>0$, $\widetilde{C}_q>0$ such that 
for all $s=(s_1,s_2,s_3)\in\R_+^3$,
$$
  \lim_{|s|\to\infty}\frac{q_1(s_1)+q_2(s_2)q_3(s_3)}{\sum_{i=1}^3 F_i(s)
	\sum_{i=1}^3 s_i + 1} = 0, \quad
	\frac{q_i(s_i)(1+q_i(s_i))}{q_i'(s_i)f_i'(s_i)(F_i(s)s_i+1)} \le C_q,
$$
$$
  \frac{\sum_{i=1}^3  \int_1^{s_i} \log (1 + q_i(v))\,dv}{\sum_{i=1}^3 
  F_i(s)\sum_{i=1}^3 s_i + 1} \le  \widetilde{C}_q.
$$
\item Weak cross-diffusion: There exists $\eta_0>0$ and $\delta\in(0,1)$ such that
for all $\eta\in[0,\eta_0]$ and $s_1$, $s_2\ge 0$,
\begin{align*}
  \bigg( & \frac{q_1'(s_1)\pa_2f_{12}(s_1,s_2)}{q_1(s_1)(1+\eta q_1(s_1))}
	+ \frac{q_2'(s_2)\pa_1f_{21}(s_1,s_2)}{q_2(s_2)(1+\eta q_2(s_2))}\bigg)^2 \\
	&\le 2(1-\delta)^2\frac{q_1'(s_1)q_2'(s_2)(f_1'(s_1)+\pa_1 f_{12}(s_1,s_2))
	(f_2'(s_2)+\pa_2 f_{21}(s_1,s_2))}{q_1(s_1)q_2(s_2)(1+\eta q_1(s_1))
	(1+\eta q_2(s_2))}.
\end{align*}

\item Initial data: $u_i^I\in L^\infty(\Omega)$ and there exists $\kappa_2>0$ such
that $u_i^I\ge\kappa_2>0$ in $\Omega$.
\end{enumerate}

\begin{remark}[Discussion of the assumptions] \label{rem.ass}\rm
We indicate where the main assumptions are needed in the existence proof.
\begin{itemize}[leftmargin=5mm]
\item Assumption (A1): The lower bound $f_i'(s)\ge\kappa_1>0$ and $f_i(0)=0$ imply 
that $f_i(s)\ge\kappa_1 s$ for all $s\ge 0$. 
This means that we require some amount of standard diffusion in the problem.
This assumption implies a uniform $L^2$ bound for the approximate solutions; see
Lemma \ref{lem.L2}.

\item Assumption (A2): This is a structure condition on the diffusion matrix.
It allows us to show that $F$ is a homeomorphism on $\R_+^3$ 
(see Lemma \ref{lem.hom}), which is needed
in the approximate scheme.

\item Assumption (A3): This condition is satisfied, for instance, for power-type
functions $q_i$ with exponent larger than or equal to one. It ensures that the entropy built out of the $q_i$ is well-behaved.

\item Assumption (A4): The conditions relate the reaction and diffusion terms.
Together with the duality estimate, they yield the uniform integrability of $Q_i$.
Note that because of this assumption, it is not possible to handle reaction terms 
which grow too fast when the unknowns become large.
The third bound is needed to show that the regularized entropy density
is bounded from below; see the arguments before \eqref{2.Qh}.

\item Assumption (A5): The weak cross-diffusion condition allows us to prove
nonlinear gradient estimates.
Expanding the square on the left-hand side of the inequality of the assumption
and choosing $\delta>0$ such that $2(1-\delta)^2<1$, we see that this assumption
implies that 
\begin{equation}\label{1.det}
  \pa_2 f_{12}(s_1,s_2)\pa_1 f_{21}(s_1,s_2) < \big(f_1'(s_1)+\pa_1 f_{12}(s_1,s_2)\big)
	\big(f_2'(s_1)+\pa_2 f_{21}(s_1,s_2)\big).
\end{equation}
It means that the determinant of the diffusion matrix $F'(u)$ is positive.
This information is needed to show that $F$ is a homeomorphism on $\R_+^3$; see
Lemma \ref{lem.hom}. Note that assumption (A5) is typically satisfied when the derivatives
of the cross diffusion terms $f_{12}$ and $f_{21}$ are assumed to be small when compared to the derivatives of the standard diffusion terms $f_1$, $f_2$; in other words when the cross diffusion is dominated in some sense by the standard diffusion. 

\item Assumption (A6): The positivity assumption on the initial data is necessary
to prove the nonnegativity of $u_i$. By using an approximation argument, we may
relax this condition to $u_i^I\ge 0$ in $\Omega$, but we leave the technical details 
to the reader.
\end{itemize}
Note that for instance, the functions
\begin{align*}
  & f_i(u_i) = \alpha_{i}u_i + u_i^{\delta}, \quad
	q_i(u_i) = u_i^\beta, \quad\beta \ge 1, \quad i=1,2,3, \\
	& f_{12}(u_1,u_2) = \alpha u_1^\gamma u_2, \quad 
	f_{21}(u_1,u_2) = \alpha u_1u_2^\gamma,
\end{align*}
satisfy Assumptions (A1)-(A5) if $\delta > 1$ is sufficiently large and
$\alpha>0$ is sufficiently small; see Lemma \ref{lem.fct} for details.
\qed
\end{remark}

The first main result is the global-in-time existence of very weak 
(i.e.\ integrable) solutions to equations \ref{1.eq}, for a given $\eps>0$.

\begin{theorem}[Global existence of solutions]\label{thm.ex}
Let $\Omega$ be a bounded open subset of $\R^d$ with a smooth boundary, let assumptions (A1)-(A6) hold, and let $\eps>0$, $T>0$. Then there exists a very weak
solution $u_i^\eps\in L^2(\Omega_T)$ to \eqref{1.eq}-\eqref{1.bic} such that
$$ 
  u_i^\eps\ge 0\mbox{ in }\Omega_T, \quad F_i(u^\eps),\,Q_i(u^\eps)\in L^1(Q_T), 
$$
and for all
$\phi_i\in H^1(0,T;L^2(\Omega))\cap L^\infty(0,T;W^{2,\infty}(\Omega))$ 
with $\na\phi_i\cdot\nu=0$ on $\pa\Omega$,
\begin{equation}\label{1.weak}
\begin{aligned}
  -\int_0^\infty\int_\Omega & u_i^\eps\pa_t\phi_i dxdt 
	- \int_0^\infty\int_\Omega F_i(u^\eps)\Delta\phi_i^\eps dxdt \\
	&= \int_0^\infty\int_\Omega Q_i(u^\eps)\phi_i dxdt
	+ \int_\Omega u_i^I(x)\phi_i(x,0)dx, 
\end{aligned}
\end{equation}
(and $u_i^\eps(0)=u_i^I$ in $H^{m}(\Omega)'$, $i=1,2,3$, where $m>2+d/2$).
Moreover, this solution satisfies the entropy inequality
\begin{equation}\label{1.ent}
  \int_\Omega h(u^\eps(t))dx 
	+ \delta\,\int_0^t\int_\Omega\sum_{i=1}^3 |\na [J_i(u_i^\eps)]|^2dxd\sigma
	\le \int_\Omega h(u^I)dx,
\end{equation}
where $h(u^\eps)$ is the entropy given by \eqref{1.h},
$\delta$ is defined in Assumption (A5), and
\begin{equation}\label{1.J}
  J_i(s) = \int_0^s\min\bigg\{1,\bigg(\frac{q'_i(y)f'_i(y)}{
	q_i(y)(1+q_i(y))}\bigg)^{1/2}\bigg\}dy, \quad s\ge 0, \quad i=1,2,3.
\end{equation}
\end{theorem}

The proof is based on a regularization procedure, entropy estimates, and
a duality method. More precisely, we replace the time derivative by the
implicit Euler discretization with time step size $\tau>0$ and regularize
the reactions $Q_i$ with parameter $\eta>0$ to make them bounded (say, $Q_i^\eta$). 
The existence
of solutions $u_i^k$, which approximate $u_i(\cdot,k\tau)$, 
is shown by techniques similarly as in \cite{DLMT15}.
Using a regularized version of the entropy \eqref{1.h}, $h^\eta$, we derive the
discrete entropy inequality (see Lemma \ref{lem.ent})
\begin{equation}\label{1.disent}
\begin{aligned}
  \int_\Omega h^\eta(u^k)dx &+ \tau\sum_{j=1}^k\sum_{i=1}^3\int_\Omega
	|\na [J_i(u_i^j)]|^2 dx \\
	&{}+ \frac{\tau}{\eps}\sum_{j=1}^k\int_\Omega Q^\eta(u^\eps)\cdot(h^\eta)'(u^\eps)dx
	\le \int_\Omega h(u^I)dx.
\end{aligned}
\end{equation}
This gives {\it{a priori}}
estimates independent of the regularization parameters $\eta$ and $\tau$
as well as the relaxation time $\eps$. Further $L^2$ bounds are obtained from 
the duality method of \cite{PiSc97}, here in the discrete version of 
\cite[Lemma 2.12]{DLMT15}. Thanks to the discrete Aubin-Lions lemma of \cite{DrJu12},
we obtain the relative compactness of the sequence of approximate solutions.
This allows us to perform the limit $(\eta,\tau)\to 0$ in the approximate problem.
\medskip 

The second main result is the fast-reaction limit.

\begin{theorem}[Fast-reaction limit]\label{thm.fast}
Let $\Omega$ be a bounded open subset of $\R^d$ with a smooth boundary, $T>0$, and let assumptions (A1)-(A6) hold. We suppose that
$q_1(u_1^I)=q_2(u_2^I)q_3(u_3^I)$ in $\Omega$, and that the functions
\begin{equation}\label{1.a}
  \R_+^2\to\R, \quad (u_2,u_3)\mapsto \frac{1}{u_i}q_1^{-1}(q_2(u_2)q_3(u_3)),
	\quad i=2,3,
\end{equation}
are continuous. Furthermore, let $u^\eps=(u_1^\eps,u_2^\eps,u_3^\eps)$ be the
very weak solution to \eqref{1.eq}-\eqref{1.bic} constructed in Theorem \ref{thm.ex}.
Then there exists a subsequence, which is not relabeled, such that, as $\eps\to 0$,
$$
  u_i^\eps\to u_i\quad\mbox{strongly in }L^1(Q_T),\ i=1,2,3.
$$
The limit $u_i \in L^1(Q_T)$ is a very weak solution to the system 
\begin{align}
  & \pa_t(u_1+u_2) = \Delta(F_1(u)+F_2(u)), \quad
	\pa_t(u_1+u_3) = \Delta(F_1(u)+F_3(u)), \label{1.eps1} \\
	& q_1(u_1) = q_2(u_2)q_3(u_3)\quad\mbox{in }Q_T. \label{1.eps2}
\end{align}
Moreover, it satisfies the entropy inequality
\begin{equation}\label{1.ent2}
  \int_\Omega h(u(t))dx 
	+ \delta\,\int_0^t\int_\Omega\sum_{i=1}^3 |\na J_i(u_i)|^2dxd\sigma
	\le \int_\Omega h(u^I)dx,
\end{equation}
where $h$ and $J_i$ are defined in \eqref{1.h} and \eqref{1.J}, respectively.
\end{theorem}

The proof is based on the following ideas: From the entropy inequality (see the discrete version \eqref{1.disent}), we
deduce immediately that 
\begin{equation}\label{1.u1}
  q_1(u_1^\eps)-q_2(u_2^\eps)q_3(u_3^\eps) \to 0\quad\mbox{strongly in }L^1(Q_T). 
\end{equation}
Here, we need the condition $q_1(u_1^I)=q_2(u_2^I)q_3(u_3^I)$, which prevents
a boundary (more precisely, an initial) layer.
We cannot directly apply the Aubin-Lions lemma to $u_i^\eps$, since the
bounds for $\pa_t u_i^\eps$ depend on $\eps$. However, $\pa_t(u_1^\eps+u_2^\eps)$
and $\pa_t(u_1^\eps+u_3^\eps)$ are uniformly bounded, showing, together
with the gradient estimate from \eqref{1.ent}, that $u_1^\eps+u_2^\eps\to v$
and $u_1^\eps+u_3^\eps \to w$ strongly in $L^1(Q_T)$. The key idea is to prove that
the mapping $(u_2,u_3)\mapsto (u_1+u_2,u_1+u_3)$ 
can be inverted (see Lemma \ref{lem.inv}). 
For this argument, we need the continuity of the functions in \eqref{1.a}.
We deduce that $u_2^\eps\to u_2$ and $u_3^\eps\to u_3$
a.e.\ in $Q_T$ and consequently $u_1^\eps=q_1^{-1}(q_2(u_2^\eps)q_3(u_3^\eps))\to u_1$
a.e.\ in $Q_T$. 

The paper is organized as follows. Theorems \ref{thm.ex} and \ref{thm.fast}
are proved in Sections \ref{sec.ex} and \ref{sec.fast}, respectively.
In the appendix, we collect some technical and auxiliary results.


\section{Proof of the existence result}\label{sec.ex}

We start here the 
\medskip

{\it{Proof of Theorem \ref{thm.ex}}}: We first show the existence of solutions
to an approximate problem. Let $T>0$, $N\in\N$, $\tau=T/N$, and
$\eta\in(0,1)$. We assume throughout this section that Assumptions (A1)-(A6) hold.
Given $u^{k-1}=(u_1^{k-1},u_2^{k-1},u_3^{k-1})\in L^\infty(\Omega;\R_+^3)$,
we wish to solve the following implicit Euler scheme with bounded reaction terms:
\begin{equation}\label{2.approx}
  \tau^{-1}(u_i^k-u_i^{k-1}) - \Delta F_i(u^k) = Q_i^\eta(u^k)\quad\mbox{in }
	\Omega,\ i=1,2,3,
\end{equation}
together with the no-flux boundary conditions
\begin{equation}\label{2.bc}
  \na F_i(u^k)\cdot\nu = 0\quad\mbox{on }\pa\Omega,\ i=1,2,3.
\end{equation}
When $k=1$, we set $u^{k-1}=u^I$. The regularized reaction terms are defined by
$$
  Q_i^\eta(u^k) = \frac{\sigma_i}{\eps}
	\bigg(\frac{q_1(u_1^k)}{1+\eta q_1(u_1^k)}
	- \frac{q_2(u_2^k)}{1+\eta q_2(u_2^k)}
	\frac{q_3(u_3^k)}{1+\eta q_3(u_3^k)}\bigg),
$$
where $\sigma_1=-1$ and $\sigma_2=\sigma_3=1$. 
They satisfy the following properties. First, a straightforward estimation gives
\begin{equation}\label{2.Q1}
  |Q_i^{\eta}(s)|\le \frac{1}{\eps}\frac{1+\eta}{\eta^2} \le \frac{2}{\eps\eta^2}
	=: K_1(\eps,\eta)	\quad\mbox{for all }s_i\ge 0,\ i=1,2,3.
\end{equation}
Second, let $Q_{i}^\eta=Q_{i,+}^\eta-Q_{i,-}^\eta$, where $Q_{i,+}^\eta\ge 0$
and $Q_{i,-}^\eta\ge 0$. Then there exists a constant $K_2(\eps,\eta)>0$ such that
\begin{equation}\label{2.Q2}
  Q_{i,-}^\eta(s)\le K_2(\eps,\eta)s_i \quad\mbox{for all }s_i\ge 0,\ i=1,2,3.
\end{equation}
Indeed, this estimate is clear for large values of $s_i$ because of the boundedness
of $Q_{i,-}^\eta$; for small values of $s_i$, it follows from $q_i(s_i)=q_i(0)
+ q'_i(\xi)s_i = q'_i(\xi)s_i$ for some $0<\xi<s_i$.

We also need the following property of $F$ (shown in Lemma \ref{lem.hom}
in the appendix): $F:\R_+^3\to\R_+^3$ is a $C^1$-diffeomorphism from $(\R_+^*)^3$
into itself and a homeomorphism from $\R_+^3$ into itself.

\subsection{Existence for scheme \eqref{2.approx}.}

We prove that there exists a strong solution to \eqref{2.approx}-\eqref{2.bc} (under the assumptions of Theorem \ref{thm.ex}).

\begin{lemma}\label{lem.approx}
Let $0<\tau\le 1/K_1(\eps,\eta)=\eps\eta^2/2$.
Then there exists a solution $u^k\in C^0(\overline\Omega;\R^3)$ to
\eqref{2.approx}-\eqref{2.bc} such that $F(u^k)\in W^{2,p}(\Omega;$ $\R^3)$
for all $p<\infty$.
\end{lemma}

\begin{proof}
The proof is a modification of the proof of Theorem 2.5 in \cite{DLMT15}.
Since our estimates are partially different, we present
a full proof. The idea is to define a fixed-point operator whose compactness
follows from the compactness of an elliptic solution operator.

{\em Step 1: Definition of the fixed-point operator.}
Let $K_{p,\Omega}>0$ be the elliptic regularity constant defined in Lemma
\ref{lem.regul1}. Introduce for $u=(u_1,u_2,u_3)\in L^\infty(\Omega;\R^3)$,
$$
  \overline{M}(u) := \max\bigg\{\tau K_{p,\Omega},\frac{1}{\kappa_1}\bigg(
	1+\tau\max_{i=1,2,3}\sup_\Omega\frac{Q_{i,-}^\eta(u)}{|u_i|}\bigg)\bigg\}.
$$
The constant $\kappa_1$ is defined in Assumption (A1).
Because of \eqref{2.Q2}, $Q_{i,-}^\eta(u)/|u_i|$ is finite and so, $\overline{M}(u)$
is finite too. 

We define the fixed-point operator $\Lambda:[0,1]\times
L^\infty(\Omega;\R^3)\to L^\infty(\Omega;\R^3)$ by
$$
  \Lambda(\sigma,u) = \Phi\circ(\sigma\Theta)\circ\Psi(u),
$$
where (recall that $F_i$ and $Q^\eta$ are continuously extended to $\R^3$ 
by setting $F_i(u)=F_i(|u_1|,|u_2|,$ $|u_3|)$, 
$Q^\eta_i(u)=Q^\eta_i(|u_1|,|u_2|,|u_3|)$),
\begin{align*}
  & \Psi:L^\infty(\Omega;\R^3)\to L^\infty(\Omega;\R_+^3)
	\times(\tau K_{p,\Omega},\infty), \\
	&\phantom{xxx}
	\Psi(u) = \big(u^{k-1}+\overline{M}(u)F(u)-u+\tau Q^\eta(u),\overline{M}(u)\big), \\
	& \Theta: L^\infty(\Omega;\R_+^3)\times(\tau K_{p,\Omega},\infty)
	\to L^\infty(\Omega;\R_+^3), \\
	&\phantom{xxx}
	\Theta(u,M) = (M\operatorname{Id}-\tau\Delta)^{-1}u
	\mbox{ with no-flux boundary conditions}, \\
	& \Phi: L^\infty(\Omega;\R_+^3)\to L^\infty(\Omega;\R_+^3), \quad
	\Phi(u) = F^{-1}(u).
\end{align*}
A computation shows that any fixed point of $\Lambda(\sigma,\cdot)$ solves
\begin{equation}\label{2.fpe}
  \overline{M}(u)F_i(u) - \tau\Delta F_i(u) 
	= \sigma\big(u_i^{k-1} + \overline{M}(u)F_i(u) - u_i + \tau Q_i^\eta(u)\big),
	\quad i=1,2,3,
\end{equation}
which means that for $\sigma=1$, this fixed point solves 
\eqref{2.approx}-\eqref{2.bc}. 

We have to show that the functions $\Phi$,
$\Theta$, and $\Psi$ are well defined.
Indeed, by Assumption (A1), the definition of $\overline{M}(u)$, and the
property $F_i(u)\ge\kappa_1 |u_i|$,
\begin{align*}
  & \overline{M}(u) F_i(u)-u_i+\tau Q_i^\eta(u) \\
	&\ge |u_i|\bigg(1 + \tau\max_{i=1,2,3}\sup_\Omega\frac{Q_{i,-}^\eta(u)}{|u_i|}\bigg)
	- |u_i| + \tau\big(Q_{i,+}^\eta(u)-Q_{i,-}^\eta(u)\big) 
	\ge \tau Q_{i,+}^\eta(u) \ge 0.
\end{align*}
We deduce that $\Psi$ is well defined. If $\Theta(u,M)=v$ for some
$u\in L^\infty(\Omega;\R_+^3)$ then $Mv-\tau\Delta v=u\ge 0$ in $\Omega$
and $\na v\cdot\nu=0$ on $\pa\Omega$. Using 
$v^-=\min\{0,v\}$ as a test function in the weak formulation of this
elliptic equation, we see that $v\ge 0$ in $\Omega$. Furthermore,
with the test function $(v-\mu)^+=\max\{0,v-\mu\}$, where 
$\mu=\|u\|_{L^\infty(\Omega)}/M$, it follows that
$$
  \tau\int_\Omega|\na(v-\mu)^+|^2 dx
	= \int_{\{v>\mu\}}(u-Mv)(v-\mu)^+ dx 
	\le 0,
$$
and hence, $\|v\|_{L^\infty(\Omega)}\le\|u\|_{L^\infty(\Omega)}/M$.
This shows that $\Theta$ is well defined. Finally, $\Phi$ is well defined
since $F$ is a homeomorphism on $\R^3_+$; see Lemma \ref{lem.hom}.

We check the properties of $\Lambda$ needed to apply the Leray-Schauder
fixed-point theorem. Clearly, $\Lambda(0,u)=0$ for all $u\in L^\infty(\Omega;\R^3)$.
The continuity of $\Lambda$ follows from the continuity of the functions
$\Psi$, $\Theta$, and $\Phi$ proved in Lemma 2.6 in \cite{DLMT15}. 
By Lemma \ref{lem.regul1} of the appendix, $\Theta(u,M)\in W^{2,p}(\Omega;\R^3)$
for any $p<\infty$. Since the embedding $W^{2,p}(\Omega)\hookrightarrow
L^\infty(\Omega)$ is compact for $p>d/2$, 
we deduce that $\Theta:L^\infty(\Omega;\R_+^3)
\times(\tau K_{p,\Omega},\infty)\to L^\infty(\Omega;\R_+^3)$ is compact too,
and the same holds for $\Lambda$.

It remains to show a uniform $L^\infty$ estimate for any fixed point $u$
(that is, such that $\Lambda(\sigma,u)=u$).
Note that any fixed point is nonnegative and thus, $F_i(u)\ge 0$ in $\Omega$.

{\em Step 2: $L^1$ estimate for $F_i(u)$.}
We claim that there exists a constant $C>0$, depending
on $\|u^{k-1}\|_{L^\infty(\Omega)}$ and $\Omega$, such that
\begin{equation}\label{2.F}
  0\le\sum_{i=1}^3\int_\Omega F_i(u)dx\le C.
\end{equation}
Indeed, the fixed point $u$ solves \eqref{2.fpe} in $\Omega$ and $\na F_i(u)\cdot\nu=0$
on $\pa\Omega$. Summing \eqref{2.fpe} for $i=1,2,3$ and denoting
$|F(u)|_1 := \sum_{i=1}^3 F_i(u)$, \eqref{2.fpe} leads to
\begin{equation}\label{2.sum}
  (1-\sigma)\overline{M}(u)|F(u)|_1 + \sigma|u|_1 - \tau\Delta|F(u)|_1
	= \sigma|u^{k-1}|_1 + \sigma\tau |Q^\eta(u)|_1.
\end{equation}
Multiplying this equation by $|F(u)|_1$ and integrating over $\Omega$ yields
\begin{equation}\label{2.aux}
\begin{aligned}
  (1-\sigma)&\int_\Omega\overline{M}(u)|F(u)|_1^2 dx
	+ \tau\int_\Omega|\na|F(u)|_1|^2 dx 
	+ \sigma\int_\Omega|u|_1|F(u)|_1 dx \\
	&= \sigma\int_\Omega|u^{k-1}|_1|F(u)|_1 dx 
	+ \sigma\tau\int_\Omega|Q^\eta(u)|_1|F(u)|_1 dx.
\end{aligned}
\end{equation}
If $\sigma=0$, we have $F_i(u)=0$ for $i=1,2,3$, and $u=0$. Therefore we consider
 $\sigma\neq 0$. Neglecting the first two integrals in \eqref{2.aux} 
and dividing this equation by $\sigma$, it follows that
$$
  \int_\Omega|u|_1|F(u)|_1 dx
	\le \int_\Omega|u^{k-1}|_1|F(u)|_1 dx 
	+ \tau\int_\Omega|Q^\eta(u)|_1|F(u)|_1 dx.
$$
Then, by \eqref{2.Q1} and $\tau K_1(\eps,\eta)\le 1$,
\begin{align*}
  \int_\Omega|u|_1|F(u)|_1 dx
	&\le 3\|u^{k-1}\|_{L^\infty(\Omega)}\int_\Omega|F(u)|_1 dx
	+ 3\tau K_1(\eps,\eta)\int_\Omega|F(u)|_1 dx \\
	&\le 3\big(\|u^{k-1}\|_{L^\infty(\Omega)}+1\big)\int_\Omega|F(u)|_1 dx.
\end{align*}

Let $R>0$. If $|u|_1\le R$, the continuity of $F_i$ gives $|F(u)|_1\le \omega(R)$,
where $\omega$ is a modulus of continuity. Therefore,
\begin{align*}
   \int_\Omega|u|_1|F(u)|_1 dx
	&\le 3\big(\|u^{k-1}\|_{L^\infty(\Omega)}+1\big)
	\bigg(\int_{\{|u|_1>R\}}|F(u)|_1 dx + \int_{\{|u|_1\le R\}}|F(u)|_1 dx\bigg) \\
	&\le 3\big(\|u^{k-1}\|_{L^\infty(\Omega)}+1\big)
	\bigg(\frac{1}{R}\int_\Omega |u|_1|F(u)|_1 dx + \omega(R)|\Omega|\bigg),
\end{align*}
where $|\Omega|$ denotes the measure of $\Omega$.
We choose $R=6(\|u^{k-1}\|_{L^\infty(\Omega)}+1)$ and obtain
$$
  \frac12\int_\Omega|u|_1|F(u)|_1 dx 
	\le 3\big(\|u^{k-1}\|_{L^\infty(\Omega)}+1\big)\omega(R)|\Omega|
	= \frac{R}{2}\omega(R)|\Omega|.
$$
We use this estimate in
\begin{align*}
  \int_\Omega|F(u)|_1 dx
	&= \int_{\{|u|_1>R\}}|F(u)|_1 dx + \int_{\{|u|_1\le R\}}|F(u)|_1 dx \\
	&\le \frac{1}{R}\int_\Omega|u|_1|F(u)|_1 dx + \omega(R)|\Omega|
  \le 2\omega(R)|\Omega|,
\end{align*}
which proves \eqref{2.F}.

{\em Step 3: $L^\infty$ estimate for $u_i$.}
We use estimate \eqref{2.Q1} for $Q^\eta$ in \eqref{2.sum}:
$$
  -\tau\Delta|F(u)|_1
  \le (1-\sigma)\overline{M}(u)|F(u)|_1 + \sigma|u|_1 - \tau\Delta|F(u)|_1
	\le \sigma|u^{k-1}|_1 + 3\sigma\tau K_1(\eps,\eta).
$$
As the right-hand side is in $L^\infty(\Omega)$, 
we can apply Lemma \ref{lem.regul2} in the appendix to conclude that
$$
  \||F(u)|_1\|_{L^\infty(\Omega)} 
	\le C\bigg(\frac{3}{\tau}\|u^{k-1}\|_{L^\infty(\Omega)} + 3K_1(\eps,\eta)
	+ \||F(u)|_1\|_{L^1(\Omega)}\bigg).
$$
Then, taking into account Assumption (A1) and \eqref{2.F}, 
$$
  \|u_i\|_{L^\infty(\Omega)}\le \kappa_1\||F(u)|_1\|_{L^\infty(\Omega)}
	\le C(\eps,\eta,\Omega,\|u^{k-1}\|_{L^\infty(\Omega)}),
$$
which shows the desired estimate (uniform with respect to the considered fixed points). 

Hence, we can apply the Leray-Schauder theorem and infer the existence of
a solution to \eqref{2.approx}-\eqref{2.bc}.

It remains to verify the continuity of $u_i^k$. We know that 
$F_i(u^k)\in W^{2,p}(\Omega)$ for all $p<\infty$. Thus, choosing $p>d/2$, 
$F_i(u^k)\in C^0(\overline\Omega)$.
Since $F$ is a homeomorphism, we conclude that $u^k_i\in C^0(\overline\Omega)$.
\end{proof}


\subsection{A priori estimates for scheme \eqref{2.approx}.}

We show several {\it{a priori}} estimates which are (except for Lemma \ref{lem.pos} below) uniform in $\eta$ and $\tau$. 
Some of these estimates are also uniform with respect to $\eps$ and will be used in
Section \ref{sec.fast}. 
We denote by $C(\delta_1,\ldots,\delta_n)$ a generic positive constant depending 
on the parameters $\delta_1,\ldots,\delta_n$, whose value may change from
occurence to occurence. We begin with

\begin{lemma}[Positivity of $u_i^k$]\label{lem.pos}
Let $\tau<1$. Then there exists a constant $\delta(\eps,\eta,\tau)>0$ depending on 
$\eps$, $\eta$, and $\tau$ such that
\begin{equation}\label{2.delta}
  u_i^k\ge\delta(\eps,\eta,\tau)
	\quad\mbox{in }\Omega, 
	\ i=1,2,3.
\end{equation}
\end{lemma}

\begin{proof}
We proceed by induction. By Assumption (A6),
$u_i^I\ge\kappa_2>0$ in $\Omega$. Let $u_i^{k-1}\ge\gamma>0$ for some $\gamma>0$, 
$i=1,2,3$. Using \eqref{2.Q2}, we find that
$$
  Q_i^\eta(u^k) = Q_{i,+}^\eta(u^k) - Q_{i,-}^\eta(u^k)
	\ge -K_2(\eps,\eta)u_i^k.
$$
Thus, choosing $M\ge (\tau K_2(\eps,\eta)+1)/\kappa_1$,
$$
  Q_i^\eta(u^k) + \frac{M}{\tau}F_i(u^k) 
	\ge Q_i^\eta(u^k) + \frac{\kappa_1 M}{\tau} u_i^k \ge \frac{u_i^k}{\tau},
$$
and consequently, using the scheme and the induction hypothesis,
$$
  \frac{M}{\tau}F_i(u^k) - \Delta F_i(u^k) 
	= \frac{u_i^{k-1}}{\tau} - \frac{u_i^k}{\tau} + Q_i^\eta(u^k) 
	+ \frac{M}{\tau}F_i(u^k)
	\ge \frac{u_i^{k-1}}{\tau} \ge u_i^{k-1}\ge \gamma.
$$
By the minimum principle (see Step 1 of the proof of Lemma \ref{lem.approx}
for the same argument), $F_i(u^k)\ge\tau\gamma/M$ for $i=1,2,3$. We know from
the proof of Lemma \ref{lem.approx} that $u_i^k\in L^\infty(\Omega)$.
By Assumption (A1), this gives $F_i(u^k)\le Cu_i^k$, 
where $C>0$ depends on the $L^\infty$ bound of $u^k$.
We infer that $u_i^k\ge\tau\gamma/(CM)=:\delta(\eps,\eta,\tau)>0$. 
\end{proof}

\begin{lemma}[Uniform $L^2$ estimate]\label{lem.L2}
There exists a constant $C>0$ independent of $\eps$, $\eta$, and $\tau$ such that
$$
  \tau\sum_{k=1}^N\int_\Omega\sum_{i=1}^3 F_i(u^k)\sum_{i=1}^3 u_i^k dx \le C, \quad
	\tau\sum_{k=1}^{N}\|u_i^k\|_{L^2(\Omega)}^2 \le C, \quad i=1,2,3.
$$
\end{lemma}

\begin{proof}
These bounds are a consequence of the duality estimate stated in Lemma \ref{lem.dual}.
Indeed, we set $v^k=2u_1^k+u_2^k+u_3^k$ and
$$
  \mu^k = \frac{2F_1(u^k)+F_2(u^k)+F_3(u^k)}{2u_1^k+u_2^k+u_3^k}.
$$
Note that $\mu^kv^k\in H^2(\Omega)$. Then $(v^k-v^{k-1})/\tau=\Delta(\mu^k v^k)$, 
since the weighted sum of the reaction terms
vanishes. Consequently, the following estimates do not depend on $\eps$ nor $\eta$. 
Lemma \ref{lem.dual} gives
\begin{equation}\label{2.aux2}
  \tau\sum_{k=1}^{N}\int_\Omega\mu^k(v^k)^2 dx
	\le C\bigg(1 + \tau\sum_{k=1}^{N}\int_\Omega\mu^k dx\bigg),
\end{equation}
where $C>0$ depends only on $u^I$, $\Omega$, and $T=N\tau$.

It remains to estimate $\int_\Omega\mu^kdx$.
Let $L>0$ and define $S=\{s=(s_1,s_2,s_3)\in\R_+^3:2s_1+s_2+s_3\le L\}$ and
\begin{align*}
  \mu(L) 
	& := \sup_{s\in S}\frac{2F_1(s)+F_2(s)+F_3(s)}{2s_1+s_2+s_3} \\
	&= \sup_{s\in S}\frac{2s_1(g_1(s_1)+g_{12}(s_1,s_2))+s_2(g_2(s_2)+g_{21}(s_1,s_2))
	+ s_3g_3(g_3)}{2s_1+s_2+s_3}.
\end{align*}
Clearly, $\mu(L)$ is finite. It follows that
$$
  \int_\Omega\mu^k dx = \int_{\{v^k\le L\}}\mu^k dx
	+ \int_{\{v^k>L\}}\mu^kdx
	\le \mu(L)|\Omega| + \frac{1}{L^2}\int_\Omega\mu^k(v^k)^2 dx.
$$
Inserting this estimate into \eqref{2.aux2} and using $k\tau\le T$, we arrive at
$$
  \bigg(1-\frac{C}{L^2}\bigg)\tau\sum_{k=1}^{N}\int_\Omega\mu^k(v^k)^2dx
	\le C\big(1+T\mu(L)|\Omega|\big).
$$
Choosing $L>0$ sufficiently large, this yields the first estimate in the statement.
For the $L^2$ bound, we observe that $F_i(u^k)\ge\kappa_1 u_i^k$ so that,
for all $j=1,2,3$,
$$
  \tau\kappa_1\sum_{k=1}^{N}\int_\Omega (u_j^k)^2 dx 
	\le \tau\sum_{k=1}^{N}\int_\Omega\sum_{i=1}^3 F_i(u^k)\sum_{i=1}^3 u_i^k dx 
	\le C(L,T,\Omega),
$$
which concludes the proof.
\end{proof}

We now introduce the regularized entropy density
$$
  h^\eta(u^k) = \sum_{i=1}^3\int_1^{u_i^k}\log \bigg(\frac{q_i(s)}{1+\eta q_i(s)} 
	\bigg)\, ds, \quad	\eta>0.
$$
We need to show that $\int_{\Omega}h^\eta dx$ is bounded from below uniformly in $\eta$
(since otherwise, the following estimates would depend on $\eta$). Indeed, we have
$h^\eta(u)=h(u)-\sum_{i=1}^3\int_1^u\log(1+\eta q_i(s))ds$, 
and $\int_{\Omega}h(u)\,dx$ is bounded from below.
Now, for $\eta \in [0,1]$, by Assumption (A4) and Lemma \ref{lem.L2}, 
it holds that (for $u_i \ge 1$ for all $i$)
\begin{align*}
  \sum_{i=1}^3 \int_{\Omega}\int_1^{u_i}\log(1+\eta q_i(s))dsdx 
	&\leq \sum_{i=1}^3 \int_{\Omega}\int_1^{u_i} \log(1 
	+ q_i(s))\, dsdx \\
  &\leq \widetilde{C}_q \int_{\Omega} \sum_{i=1}^3F_i(u)\, \sum_{i=1}^3 u_i\, dx
  \leq C. 
\end{align*}
This means that the integral $\int_{\Omega}h^\eta dx$ is bounded
from below uniformly in $\eta$, showing our claim.
Since $q_i'\ge 0$, the function $h^\eta$ is convex.
The construction of $h^\eta$ allows for the control of the reaction terms since
\begin{align}
  Q^\eta(u^k)\cdot (h^\eta)'(u^k)
	&= -\frac{1}{\eps}\bigg(\frac{q_1(u_1^k)}{1+\eta q_1(u_1^k)}
	- \frac{q_2(u_2^k)q_3(u_3^k)}{(1+\eta q_2(u_2^k))(1+\eta q_3(u_3^k))}\bigg) 
	\nonumber \\
	&\phantom{xx}{}\times\bigg(\log\frac{q_1(u_1^k)}{1+\eta q_1(u_1^k)}
	- \log\frac{q_2(u_2^k)q_3(u_3^k)}{(1+\eta q_2(u_2^k))(1+\eta q_3(u_3^k))}\bigg)
	\le 0. \label{2.Qh}
\end{align}

\begin{lemma}[Entropy estimate]\label{lem.ent}
Let $0<\eta\le\min\{1,\eta_0\}$, where $\eta_0$ is defined in Assumption (A5).
Then (with $\delta>0$ from Assumption (A5))
\begin{equation}\label{2.ent}
\begin{aligned}
  \int_\Omega h^\eta(u^k)dx
	&+ \delta\sum_{k=1}^N\int_\Omega\sum_{i=1}^3
	\frac{q_i'(u_i^k)f'_i(u_i^k)}{q_i(u_i^k)(1+q_i(u_i^k))}|\na u_i^k|^2 dx \\
	&{} -\tau\sum_{k=1}^N\int_\Omega Q^\eta(u^k)\cdot (h^\eta)'(u^k)dx
	\le \int_\Omega h(u^I)dx.
\end{aligned}
\end{equation}
\end{lemma}

\begin{proof}
We know from Lemma \ref{lem.pos} that $u_i^k$ is strictly positive, so
$(\pa h^\eta/\pa u_i)(u^k)$ is an admissible test function in the weak formulation
of \eqref{2.approx}:
\begin{equation}\label{2.ee}
\begin{aligned}
  \int_\Omega (u^k-u^{k-1})\cdot (h^\eta)'(u^k)dx
	&+ \tau\int_\Omega\na u^k:(h^\eta)''(u^k)F'(u^k)\na u^k dx \\
	&= \tau\int_\Omega Q^\eta(u^k)\cdot (h^\eta)'(u^k)dx,
\end{aligned}
\end{equation}
where ``:'' is the Frobenius matrix product. 
Summing \eqref{2.ee} over $k=1,\ldots,j$
and taking into account \eqref{2.Qh} and 
$$
  \int_\Omega\big(h^\eta(u^k)-h^\eta(u^{k-1})\big)dx 
	\le \int_\Omega (u^k-u^{k-1})\cdot (h^\eta)'(u^k)dx,
$$
which follows from the convexity of $h^\eta$, we find that
\begin{equation}\label{2.aux3}
\begin{aligned}
  \int_\Omega h^\eta(u^j)dx &+ \tau\sum_{k=1}^j\int_\Omega
	\na u^k:(h^\eta)''(u^k)F'(u^k)\na u^k dx \\
	&{}- \tau\sum_{k=1}^j\int_\Omega Q^\eta(u^k)\cdot (h^\eta)'(u^k)dx 
	\le \int_\Omega h^\eta(u^I)dx.
\end{aligned}
\end{equation}
A straightforward computation yields
\begin{align*}
  \int_\Omega & \na u^k:(h^\eta)''(u^k)F'(u^k)\na u^k dx \\
	&= \int_\Omega\big(T_1|\na u_1^k|^2 + T_2|\na u_2^k|^2 + T_3|\na u_3^k|^2
	+ T_4\na u_1^k\cdot\na u_2^k\big)dx,
\end{align*}
where 
\begin{align*}
  T_1 &= \frac{q_1'(u_1^k)}{q_1(u_1^k)(1+\eta q_1(u_1^k))}
	\big(f_1'(u_1^k)+\pa_1 f_{12}(u_1^k,u_2^k)\big),\\
	T_2 &= \frac{q_2'(u_2^k)}{q_2(u_2^k)(1+\eta q_2(u_2^k))}
	\big(f_2'(u_2^k)+\pa_2 f_{21}(u_1^k,u_2^k)\big),\\
	T_3 &= \frac{q_3'(u_3^k)}{q_3(u_3^k)(1+\eta q_3(u_3^k))}f'_3(u_3^k) , \\
	T_4 &= \frac{q_1'(u_1^k)}{q_1(u_1^k)(1+\eta q_1(u_1^k))}\pa_2 f_{12}(u_1^k,u_2^k)
	+ \frac{q_2'(u_2^k)}{q_2(u_2^k)(1+\eta q_2(u_2^k))}\pa_1 f_{21}(u_1^k,u_2^k). 
\end{align*}
Set $\alpha=(1-\delta)T_1/T_4$, where $\delta$ comes from Assumption (A5).
Then Young's inequality and $T_4^2\le 2(1-\delta)^2 T_1T_2$ (see Assumption (A5)) 
show that, for $0<\eta\le\eta_0$,
\begin{align*}
  T_4\na u_1^k\cdot\na u_2^k 
	&\ge -T_4\alpha|\na u_1^k|^2 - \frac{T_4}{4\alpha}|\na u_2^k|^2
	= -(1-\delta)T_1|\na u_1^k|^2 - \frac{T_4^2}{4(1-\delta)T_1}|\na u_2^k|^2 \\
	&\ge -(1-\delta)T_1|\na u_1^k|^2 - (1-\delta)T_2|\na u_2^k|^2.
\end{align*}
(Observe that Assumption (A5) could be weakened to $T_4^2\le 4(1-\delta)^2 T_1T_2$,
but then we would need to impose \eqref{1.det} as an additional constraint.)
We deduce that
$$
  \int_\Omega\na u^k:(h^\eta)''(u^k)F'(u^k)\na u^k dx
	\ge \delta\int_\Omega\big(T_1|\na u_1^k|^2+T_2|\na u_2^k|^2\big)dx
	+ \int_\Omega T_3|\na u_3^k|^2.
$$
Hence, inserting this estimate into \eqref{2.aux3}, observing that
$\eta\le 1$ and $h^\eta(u^I)\le h(u^I)$, and including the
reaction terms, the result is shown.
\end{proof}

\begin{lemma}[Estimate for the discrete time derivative]
\label{lem.time}
Let $m>2+d/2$. Then there
exists a constant $C(\eps)>0$ independent of $\eta$ and $\tau$, but depending on
$\eps$, such that
$$
   \tau\sum_{k=1}^N\big\|\tau^{-1}(u_i^k-u_i^{k-1})\big\|_{H^{m}(\Omega)'} 
	\le C(\eps), \quad i=1,2,3.
$$
\end{lemma}

\begin{proof}
Using the duality estimate from Lemma \ref{lem.L2}, we obtain 
\begin{align*}
  \tau\sum_{k=1}^N\int_\Omega  F_i(u^k)dx 
	&\le \tau\sum_{k=1}^N\int_{\{u^k\le 1\}}F_i(u^k)dx 
	+ \tau\sum_{k=1}^N\int_{\{u^k>1\}}F_i(u^k)dx \\
	&\le \tau\sum_{k=1}^N\int_{\{u^k\le 1\}}F_i(u^k)dx 
	+ \tau\sum_{k=1}^N\int_{\{u^k>1\}}F_i(u^k)u_i^kdx \le C.
\end{align*}
Furthermore, by Assumption (A4) and Lemma \ref{lem.L2} again,
\begin{align*}
  \tau\sum_{k=1}^N\int_\Omega|Q_i^\eta(u^k)| dx
	&\le \frac{\tau}{\eps}
	\sum_{k=1}^N\int_\Omega\big(q_1(u_1^k)+q_2(u_2^k)q_3(u_3^k)\big)dx \\
	&\le \frac{C\tau}{\eps}\sum_{k=1}^N\int_\Omega
	\bigg(\sum_{j=1}^3 F_j(u^k)\sum_{j=1}^3 u_j^k + 1\bigg) 
	\le C(\eps).
\end{align*}
Since $m>2+d/2$, we know that $\phi_i\in H^m(\Omega)\hookrightarrow W^{2,\infty}(\Omega)$.
Therefore, we can write
$$
  \tau\sum_{k=1}^N\tau^{-1}(u_i^k-u_i^{k-1})\phi_i dx
	= \tau\sum_{k=1}^N\int_\Omega F_i(u^k)\Delta\phi_i dx
	+ \tau\sum_{k=1}^N\int_\Omega Q_i^\eta(u^k)\phi_i dx.
$$
In view of the two previous estimates, we infer that
$$
  \tau\sum_{k=1}^N\big\|\tau^{-1}(u_i^k-u_i^{k-1})\big\|_{H^{m}(\Omega)'}
	= \sup_{\|\phi_i\|_{H^m(\Omega)}=1}
	\sum_{k=1}^N\bigg|\int_\Omega(u_i^k-u_i^{k-1})\phi_i dx\bigg| \le C(\eps),
$$
showing the desired bound.
\end{proof}

\begin{lemma}[Uniform $W^{1,1}$ estimate]\label{lem.W11}
There exists a constant $C>0$ independent of $\eps$, $\eta$, and $\tau$ such that
$$
  \tau\sum_{k=1}^N\|u_i^k\|_{W^{1,1}(\Omega)} \le C.
$$
\end{lemma}

\begin{proof}
The Cauchy-Schwarz inequality gives
\begin{align*}
  \tau\sum_{k=1}^N\|\na u_i^k\|_{L^1(\Omega)}
	&\le \bigg(\tau\sum_{k=1}^N\int_\Omega
	\frac{q_i'(u_i^k)f_i'(u_i^k)}{q_i(u_i^k)(1+q_i(u_i^k))}|\na u_i^k|^2 dx
	\bigg)^{1/2} \\
	&\phantom{xx}{}\times\bigg(\tau\sum_{k=1}^N\int_\Omega
	\frac{q_i(u_i^k)(1+q_i(u_i^k))}{q_i'(u_i^k)f_i'(u_i^k)}dx\bigg)^{1/2}.
\end{align*}
In view of Lemma \ref{lem.ent}, the first factor is bounded uniformly in
$\eps$, $\eta$, and $\tau$. By Assumption (A4) and the duality estimate
in Lemma \ref{lem.L2}, it follows that
$$
  \tau\sum_{k=1}^N\|\na u_i^k\|_{L^1(\Omega)}
	\le C\bigg(\tau\sum_{k=1}^N\int_\Omega\bigg(\sum_{i=1}^3 F_i(u^k)
	\sum_{i=1}^3 u_i^k + 1\bigg)dx\bigg)^{1/2} \le C.
$$
Taking into account the uniform $L^2$ bound, we can conclude.
\end{proof}


\subsection{Limit $(\eta,\tau)\to 0$.}\label{sec.lim}

Let $u_i^{(\tau)}(x,t)=u_i^k(x)$ for $x\in\Omega$, $t\in((k-1)\tau,k\tau]$, $i=1,2,3$,
be piecewise constant functions in time, and set $u^{(\tau)}=(u_1^{(\tau)},u_2^{(\tau)},
u_3^{(\tau)})$. We introduce the time shift operator 
$(\sigma_\tau u^{(\tau)})(x,t)=u^{k+1}(x)$ for $x\in\Omega$,
$t\in((k-1)\tau,k\tau]$. 
Let $\phi_i:(0,T)\to H^m(\Omega)$ be a piecewise constant function such that
$m>2+d/2$, $\phi_i(t)=0$ for $((N-1)\tau,N\tau]$, and $\na\phi_i\cdot\nu=0$ on 
$\pa\Omega$, $t>0$.   
Then the weak formulation of \eqref{2.approx} reads as
\begin{equation}\label{2.weak}
\begin{aligned}
  -\int_0^T\int_\Omega & u_i^{(\tau)}(\sigma_\tau\phi_i-\phi_i)dxdt
  - \int_0^T\int_\Omega F_i(u^{(\tau)})\Delta\phi_i dxdt \\
	&= \int_0^T\int_\Omega Q^\eta_i(u^{(\tau)})\phi_i dxdt
	+ \int_\Omega u_i^I(x)\phi_i(x,0)dx.
\end{aligned}
\end{equation}
Lemmas \ref{lem.time} and \ref{lem.W11} give the (uniform with respect to $\tau$, $\eta$) bounds
$$
  \tau^{-1}\|\sigma_\tau u_i^{(\tau)}-u_i^{(\tau)}\|_{L^1(0,T-\tau;H^{m}(\Omega)')}
	+ \|u_i^{(\tau)}\|_{L^1(0,T;W^{1,1}(\Omega))} \le C(\eps).
$$
Observing that the embedding $W^{1,1}(\Omega)\hookrightarrow L^1(\Omega)$
is compact, we can apply the Aubin-Lions lemma in the version of 
\cite{DrJu12} to infer the existence of a subsequence, which is not relabeled, 
such that, as $(\eta,\tau)\to 0$,
$$
  u_i^{(\tau)}\to u_i\quad\mbox{strongly in }L^1(Q_T),
$$
recalling that $Q_T=\Omega\times(0,T)$. According to Lemma \ref{lem.approx}, the limit
has to be performed in such a way that $\tau\le \eps\eta^2/2$ is verified.
Possibly for a subsequence, the convergence also holds a.e.\ in $Q_T$.
The positivity estimate from Lemma \ref{lem.pos} implies that
$u_i\ge 0$ in $Q_T$. Moreover,
$$
  F_i(u^{(\tau)})\to F_i(u), \quad Q^\eta(u^{(\tau)})\to Q(u)
	\quad\mbox{a.e. in }Q_T.
$$
Assumption (A4) implies that $(Q_i(u^{(\tau)}))$ is equi-integrable. More precisely,
we have for all $\delta>0$ the existence of $R_0>0$ such that for all $R\ge R_0$
and $|u|\ge R$, 
$$
  \frac{Q_i(u)}{\sum_{i=1}^3 F_i(u)\sum_{i=1}^3 u_i
	+1} \le \delta.
$$
Then we conclude from Lemma \ref{lem.L2} that
\begin{align*}
  \int_0^T\int_{\{u^{(\tau)} \ge R\}}Q^\eta_i(u^{(\tau)})dxdt
	&\le \int_0^T\int_{\{u^{(\tau)} \ge R\}} \big(q_1(u_1^{(\tau)})
	+ q_2(u_1^{(\tau)})q_3(u_1^{(\tau)})\big)dxdt \\
	&\le \delta\int_0^T\int_\Omega\bigg(\sum_{i=1}^3 F_i(u^{(\tau)})
	\sum_{i=1}^3 u_i^{(\tau)} + 1\bigg)dxdt \le C\delta,
\end{align*}
where $C>0$ is independent of $\eps$ (and $\tau$, $\eta$). 
This shows the equi-integrability of $(Q_i(u^{(\tau)}))$.
The same conclusion holds for $(F_i(u^{(\tau)}))$ since
\begin{equation}\label{2.FR}
  \int_0^T\int_{\{u^{(\tau)}\ge R\}}F_i(u^{(\tau)})dxdt
	\le \frac{1}{R}\int_0^T\int_\Omega F_i(u^{(\tau)})u_i^{(\tau)}dxdt
	\le \frac{C}{R}.
\end{equation}
We deduce from Vitali's convergence theorem that
$$
  F_i(u^{(\tau)})\to F_i(u), \quad 
	Q_i^\eta(u^{(\tau)})\to Q_i(u)
	\quad\mbox{strongly in }L^1(Q_T),\ i=1,2,3.
$$
Therefore, we can perform the limit $(\eta,\tau)\to 0$ in \eqref{2.weak},
showing that $u=(u_1,u_2,u_3)$ solves \eqref{1.weak}
for all $\phi_i\in H^1(0,T;L^2(\Omega))\cap
L^\infty(0,T;H^m(\Omega))$, $i=1,2,3$. By a density argument,
we see that the weak formulation also holds for all 
$\phi_i\in L^\infty(0,T;W^{2,\infty}(\Omega))$ 
with $\na\phi_i\cdot\nu=0$ on $\pa\Omega$.

It remains to verify the entropy inequality \eqref{1.ent}. 
We have from \eqref{2.ent}:
\begin{equation}\label{2.ent2}
  \int_\Omega h^\eta(u^{(\tau)})dx
	+ \delta\sum_{i=1}^3\int_0^T\int_\Omega|\na [J_i(u_i^{(\tau)})] |^2 dxdt
	\le \int_\Omega h(u^I)dx,
\end{equation}
where $J_i$ is defined in \eqref{1.J}.
The a.e.\ convergence of $u_i^{(\tau)}\to u_i$ implies that
$J_i(u_i^{(\tau)})\to J_i(u_i)$ a.e.\ in $Q_T$. Moreover, $J_i(s)\le s$, and 
thanks to
the uniform $L^2$ bound for $u_i^{(\tau)}$, we deduce that $(J_i(u_i^{(\tau)}))$
is bounded in $L^2(0,T;H^1(\Omega))$. Up to a subsequence, we have
$\na J_i(u_i^{(\tau)})\rightharpoonup \na J_i(u_i)$ weakly in $L^2(Q_T)$.
As $h^\eta$ is convex and continuous, it is weakly lower semicontinuous
\cite[Corollary 3.9]{Bre11} and
$$
  \int_\Omega h^\eta(u)dx \le \liminf_{\tau\to 0}\int_\Omega h(u^{(\tau)})dx.
$$
Since $(h^\eta)$ converges to $h$ monotonically, we infer from the
monotone convergence theorem that 
$$
  \int_\Omega h(u)dx \le \liminf_{(\eta,\tau)\to 0}\int_\Omega h(u^{(\tau)})dx.
$$
Therefore, observing that the square of the $L^2$ norm is also
weakly lower semicontinuous, we may pass to the limit $(\eta,\tau)\to 0$
in \eqref{2.ent2} to conclude \eqref{1.ent}.


\section{Proof of the fast-reaction limit}\label{sec.fast}

In the previous section, we have shown some {\it{a priori}} estimates for the
approximate solution $u_i^{(\tau)}$, which are also independent of $\eps$.
Indeed, by Lemmas \ref{lem.L2} and \ref{lem.ent}, 
\begin{align}
  & \int_0^T\int_\Omega(u_i^{(\tau)})^2 dxdt  
	+ \int_0^T\int_\Omega F_i(u^{(\tau)})u_i^{(\tau)} dxdt \le C, \label{3.est1} \\
	& \int_0^T\int_\Omega\frac{q'_i(u_i^{(\tau)})f'_i(u_i^{(\tau)})}{q_i(u_i^{(\tau)})
	(1+q_i(u_i^{(\tau)}))}|\na u_i^{(\tau)}|^2 dxdt \le C, \quad i=1,2,3,
	\label{3.est2} \\
	& \int_0^T\int_\Omega \bigg(\frac{q_1(u_1^{(\tau)})}{1+\eta q_1(u_1^{(\tau)})}
	- \frac{q_2(u_2^{(\tau)})q_3(u_3^{(\tau)})}{(1+\eta q_2(u_2^{(\tau)}))
	(1+\eta q_3(u_3^{(\tau)}))}\bigg) \nonumber \\
	&\phantom{xx}{}\times\bigg(\log\frac{q_1(u_1^{(\tau)})}{1+\eta q_1(u_1^{(\tau)})}
	- \log\frac{q_2(u_2^{(\tau)})q_3(u_3^{(\tau)})}{(1+\eta q_2(u_2^{(\tau)}))
	(1+\eta q_3(u_3^{(\tau)}))}\bigg)dxdt
	\le \eps C. \label{3.est3}
\end{align}
As mentioned in Section \ref{sec.lim}, estimates 
\eqref{3.est1} and \eqref{3.est2} yield the bound
\begin{equation}\label{3.eps1}
  \|J_i(u_i^{(\tau)})\|_{L^2(0,T;H^1(\Omega))} \le C,
\end{equation}
which is uniform in $\eps$, $\eta$, and $\tau$. We need more
 uniform bounds to be able to pass to the limit $\eps \to 0$.
 \medskip 
 
 We start here the
 \medskip
 
 {\it{Proof of Theorem \ref{thm.fast} }}: We systematically denote by 
 $u^\eps=(u_1^\eps,u_2^\eps,u_3^\eps)$ a very weak solution to 
\eqref{1.eq}-\eqref{1.bic} constructed in Theorem \ref{thm.ex}. We first state the
\medskip

\begin{lemma}[$\eps$-uniform estimates]\label{lem.eps}
There exists $C>0$ independent of $\eps$ such that, for $i=1,2,3$,
\begin{align}
  \|u_i^\eps\|_{L^2(Q_T)} + \|F_i(u^\eps)u_i^\eps\|_{L^1(Q_T)} 
	&\le C, \label{3.eps.u} \\
  \|J_i(u_i^\eps)\|_{L^2(0,T;H^1(\Omega))} + \|u_i^\eps\|_{L^1(0,T;W^{1,1}(\Omega))} 
	&\le C, \label{3.eps.k} \\
	\big\|q_i(u_1^\eps) - q_2(u_2^\eps)q_3(u_3^\eps)\big\|_{L^1(Q_T)} 
	&\le C\sqrt{\eps}. \label{3.eps.q}
\end{align}
\end{lemma}

\begin{proof}
The first estimate in \eqref{3.eps.u} follows immediately from \eqref{3.est1} 
 by performing the limit $(\eta,\tau)$ and using the weakly lower semicontinuity of 
the $L^2$ norm. The results in Subsection \ref{sec.lim} imply that
$F_i(u^{(\tau)})u_i^{(\tau)}\to F_i(u^\eps)u_i^\eps$ a.e.\ in $Q_T$. 
Then Fatou's lemma and the second estimate in \eqref{3.est1} yield
\begin{align*}
  \int_0^T\int_\Omega F_i(u^\eps)u_i^\eps dxdt
	&= \int_0^T\int_\Omega \liminf_{\tau\to 0}F_i(u^{(\tau)})u_i^{(\tau)} dxdt \\
	&\le \liminf_{\tau\to 0}\int_0^T\int_\Omega F_i(u^{(\tau)})u_i^{(\tau)} dxdt
	\le C,
\end{align*}
and this implies the second estimate in \eqref{3.eps.u}.

We have shown in the proof of Theorem \ref{thm.ex} that there exists a
subsequence of $(u_i^{(\tau)})$ (not relabeled) such that
$u_i^{(\tau)}\to u_i^\eps$ in $L^1(Q_T)$ and a.e. Since $J_i$ is continuous,
we have $J_i(u_i^{(\tau)})\to J_i(u_i^\eps)$ a.e.
Estimate \eqref{3.eps1} then implies that, up to a subsequence,
$J_i(u_i^{(\tau)})\rightharpoonup J_i(u_i^\eps)$ weakly in $L^2(0,T;H^1(\Omega))$.
Because of the weakly lower semi-continuity of the norm, we infer from \eqref{3.eps1}
that the first estimate in \eqref{3.eps.k} holds.
The $W^{1,1}$ bound for $(u_i^\eps)$ follows as in the proof of Lemma \ref{lem.W11},
using the $L^2$ bound of $\na J_i(u_i^\eps)$.

It follows from the a.e.\ convergence $q_i(u_i^{(\tau)})\to q_i(u_i^\eps)$, 
Fatou's
lemma, and estimate \eqref{3.est3} that
$$
  \int_0^T\int_\Omega \big(q_1(u_1^\eps)-q_2(u_2^\eps)q_3(u_3^\eps)\big)
	\log \bigg( \frac{q_1(u_1^\eps)}{q_2(u_2^\eps)q_3(u_3^\eps)} \bigg)\,dxdt \le \eps C.
$$
Thus, the elementary inequality $4(a^{1/2}-b^{1/2})^2\le(a-b)\log(a/b)$ gives
\begin{align*}
  \big\|q_1 & (u_1^\eps)^{1/2} - \big(q_2(u_2^\eps)q_3(u_3^\eps)\big)^{1/2}
	\big\|_{L^2(Q_T)}^2 \\
  &\le \frac14\int_0^T\int_\Omega\big(q_1(u^\eps_1)-q_2(u^\eps_2)q_3(u^\eps_3)\big)
	\log \bigg( \frac{q_1(u^\eps_1)}{q_2(u^\eps_2)q_3(u^\eps_3)} \bigg)\, dxdt \le \eps C.
\end{align*}
We now use Assumption (A4) and \eqref{3.est1} to infer that
\begin{align*}
  \big\|q_1 & (u_1^\eps) - q_2(u_2^\eps)q_3(u_3^\eps)\big\|_{L^1(Q_T)} \\
	&\le \big\|q_1(u_1^\eps)^{1/2} - \big(q_2(u_2^\eps)q_3(u_3^\eps)\big)^{1/2}
	\big\|_{L^2(Q_T)} 
	\big\|q_1(u_1^\eps)^{1/2} + \big(q_2(u_2^\eps)q_3(u_3^\eps)\big)^{1/2}
	\big\|_{L^2(Q_T)} \\
	&\le C\sqrt{\eps}\|q_1(u_1^\eps) + q_2(u_2^\eps)q_3(u_3^\eps)\big\|_{L^1(Q_T)}^{1/2}\\
	&\le C\sqrt{\eps}\bigg(\int_0^T\int_\Omega\sum_{i=1}^3 
	F_i(u^\eps)\sum_{i=1}^3 u_i^\eps dxdt\bigg)^{1/2} \le C\sqrt{\eps},	
\end{align*}
which concludes the proof.
\end{proof}

Unfortunately, the estimate on the discrete time derivative in Lemma \ref{lem.time}
is not independent of $\eps$, which prevents the direct use of the Aubin-Lions lemma.
We overcome this problem by applying this lemma to 
$u_1^\eps+u_2^\eps$ and $u_1^\eps+u_3^\eps$ and by exploiting estimate 
\eqref{3.eps.q}. Indeed, these sums solve
$$
  \pa_t(u_1^\eps+u_i^\eps) = \Delta\big(F_1(u^\eps)+F_i(u^\eps)\big), \quad i=2,3.
$$
Estimate \eqref{3.eps.u} shows that
$(F_i(u^\eps))$ is bounded in $L^1(Q_T)$.
Consequently, $\Delta F_i(u^\eps)$ is bounded in $L^1(0,T;W^{2,\infty}(\Omega)')$, 
i.e.
$$
  \|\pa_t(u_1^\eps+u_i^\eps)\|_{L^1(0,T;W^{2,\infty}(\Omega)')} \le C, \quad i=2,3.
$$
Using this estimate together with the $W^{1,1}$ bound \eqref{3.eps.k} for $u_1^\eps+u_i^\eps$, 
we can apply the
Aubin-Lions lemma of \cite{Sim87} to find a subsequence, which is not relabeled,
such that, as $\eps\to 0$,
\begin{equation}\label{3.conv1}
  u_1^\eps+u_2^\eps \to v_2, \quad u_1^\eps+u_3^\eps \to v_3 
	\quad\mbox{strongly in }L^1(Q_T).
\end{equation}
Moreover, by \eqref{3.eps.q},
\begin{equation}\label{3.conv2}
  q_1(u_1^\eps) - q_2(u_2^\eps)q_3(u_3^\eps) \to 0 \quad
	\mbox{strongly in }L^1(Q_T).
\end{equation}
We claim that these convergences are sufficient to infer the strong convergence
of $(u_i^\eps)$ in $L^1(Q_T)$ for $i=1,2,3$. To show this, we need 
the following
auxiliary result:

\begin{lemma}[Inversion of $q(u_1)=q_2(u_2)q_3(u_3)$]\label{lem.inv}
The function $g:\R_+^2\to\R_+^2$, 
$$
  g(u_2,u_3) = \begin{pmatrix}
	u_2 + q_1^{-1}(q_2(u_2)q_3(u_3)) \\
	u_3 + q_1^{-1}(q_2(u_2)q_3(u_3))
	\end{pmatrix}
$$
is a homeomorphism on $\R_+^2$.
\end{lemma}

\begin{proof}
The proof is based on Proposition 6.1 of \cite{DLMT15}. In order to use this 
proposition,  we write $g$ as $g(u_2,u_3)=(a_2(u_2,u_3)u_2,a_3(u_2,u_3)u_3)^\top$,
with
$$
  a_i(u_2,u_3) = 1 + \frac{1}{u_i}q_1^{-1}(q_2(u_2)q_3(u_3)), \quad i=2,3.
$$
Note that by Assumption (A3), the inverse of $q_1$ exists on $\R_+$ and that
by hypothesis \eqref{1.a}, $a_2$ and $a_3$ are continuous on $\R_+^2$. 
They are bounded from below, $a_i(u_2,u_3)\ge 1$ for all $(u_2,u_3)$. 
Moreover, $u\mapsto u_ia_i(u)$ is increasing in each 
variable, $g\in C^1((\R_+^*)^2;(\R_+^*)^2)$, and the determinant of its Jacobian 
is strictly positive:
$$
  \det g' = 1 + (q_1^{-1})'(q_2(u_2)q_3(u_3))
	\big(q_2'(u_2)q_3(u_3)+q_2(u_2)q_3'(u_3)\big) \ge 1.
$$
Then Proposition 6.1 in \cite{DLMT15} shows that $g$ is a homeomorphism on $\R_+^2$.
\end{proof}

We come back to the proof of Theorem \ref{thm.fast}.
We proceed with the limit $\eps\to 0$. Limit \eqref{3.conv2} and the continuity of
$q_1^{-1}$ imply that
$$
  u_1^\eps - q_1^{-1}(q_2(u_2^\eps)q_3(u_3^\eps)) \to 0
	\quad\mbox{strongly in }L^1(Q_T).
$$
We deduce from \eqref{3.conv1} that
\begin{align*}
  u_2^\eps + q_1^{-1}(q_2(u_2^\eps)q_3(u_3^\eps)) \to v_2
	&\quad\mbox{strongly in }L^1(Q_T), \\
	u_3^\eps + q_1^{-1}(q_2(u_2^\eps)q_3(u_3^\eps)) \to v_3
	&\quad\mbox{strongly in }L^1(Q_T).
\end{align*}
Clearly, all these convergences also hold a.e.\ in $Q_T$ (maybe only for a subsequence).
Lemma \ref{lem.inv} shows that $g$ is invertible and hence,
$$
  (u_2^\eps,u_3^\eps) = g^{-1}\big(u_2^\eps + q_1^{-1}(q_2(u_2^\eps)q_3(u_3^\eps)),
	u_3^\eps + q_1^{-1}(q_2(u_2^\eps)q_3(u_3^\eps))\big).
$$
We infer from the continuity of $g^{-1}$ that
$$
  (u_2^\eps,u_3^\eps) \to g^{-1}(v_2,v_3) \quad
	\mbox{a.e. in }Q_T.
$$
By \eqref{3.eps.u}, the convergence also holds in $L^p(Q_T)$ for all $p<2$. 

We set
$$
  (u_2,u_3) := g^{-1}(v_2,v_3), \quad
	u_1 := q_1^{-1}(q_2(u_2)q_3(u_3)).
$$
The uniform integrability of $q_1(u_1^\eps)$ and $q_2(u_2^\eps)q_3(u_3^\eps)$
from Assumption (A4) and Vitali's theorem now imply that
$$
  q_1(u_1^\eps)\to q_1(u_1), \quad 
	q_2(u_2^\eps)q_3(u_3^\eps)\to q_2(u_2)q_3(u_3)\quad\mbox{strongly in }L^1(Q_T).
$$
This shows that \eqref{1.eps2} holds.
Furthermore, the uniform integrability of $F_i(u^\eps)$ from \eqref{3.eps.u}
and the above convergences give
$$
  F_i(u^\eps)\to F_i(u)\quad\mbox{strongly in }L^1(Q_T), \quad i=1,2,3.
$$
We can now perform the limit $\eps\to 0$ in the equations
\begin{align*}
  \int_0^T\int_\Omega (u_1^\eps+u_i^\eps)\pa_t\phi_i dxdt
	&+ \int_0^T\int_\Omega(F_i(u^\eps)+F_i(u^\eps))\Delta\phi_i dxdt \\
	&= -\int_\Omega(u_1^I+u_i^I)(x)\phi_i(x,0)dx, \quad i=2,3,
\end{align*}
to conclude that $u_1+u_2$ and $u_2+u_3$ solve \eqref{1.eps1}.

Estimate \eqref{3.eps.k} and the strong convergence of $(u_i^\eps)$ in $L^1(Q_T)$
allow us to pass to the inferior limit $\eps\to 0$ in \eqref{1.ent} to conclude that
\eqref{1.ent2} holds. Here, we use the weakly lower semicontinuity of the
integrals as in the end of the proof of Theorem \ref{thm.ex}.
This concludes the proof of Theorem \ref{thm.fast}.

\begin{remark}[Example]\label{rem.ex}\rm
We consider $q_i(s)=s$. Then $u_1=u_2u_3$ and the limiting system becomes
$$
  \pa_t(u_2u_3+u_2) = \Delta\big(F_1(u)+F_2(u)\big), \quad 
	\pa_t(u_2u_3+u_3) = \Delta\big(F_1(u)+F_3(u)\big).
$$
Set $v=u_1+u_2=u_2+u_2u_3$ and $w=u_1+u_3=u_3+u_2u_3$. The mapping 
$(u_2,u_3)\mapsto(v,w)$ can be inverted explicitly:
\begin{align*}
	u_2(v,w) &= -\frac12(w-v+1) + \frac12\sqrt{(w-v+1)^2+4v}, \\
	u_3(v,w) &= u_2(w,v) = -\frac12(v-w+1) + \frac12\sqrt{(v-w+1)^2+4w}.
\end{align*}
It follows that $u_1 = u_2u_3 = w-u_3(v,w)$.
Then we can write
\begin{equation*}
  \pa_t v = \Delta G_2(v,w), \quad \pa_t w = \Delta G_3(v,w),
\end{equation*}
where $G_i(v,w) = (F_1+F_i)(w-u_3(v,w),u_2(v,w),u_3(v,w))$.
This system has an entropy structure. Indeed, let
\begin{align*}
  h_0(v,w) &= (w-u_3(v,w))\big(\log(w-u_3(v,w))-1\big)
	+ u_2(v,w)\big(\log u_2(v,w)-1\big) \\
	&\phantom{xx}{}+ u_3(v,w)\big(\log u_3(v,w)-1\big).
\end{align*}
Then
$$
  \int_\Omega h_0(v(t),w(t))dx
	+ \delta\sum_{i=1}^3\int_0^t\int_\Omega|\na [J_i(u(v,w))]|^2 dxds
	\le \int_\Omega h_0(u^I)dx,
$$
where $J_i$ is given by \eqref{1.J}. For example, if $f_i(s)=s$, we 
can compute $J_i$ explicitly. If $s\ge s_0:=(-1+\sqrt{5})/2$, we have
$1/\sqrt{s(1+s)}\le 1$ and thus,
\begin{align*}
  J_i(s) &= \int_0^s\min\bigg\{1,\frac{1}{\sqrt{y(1+y)}}\bigg\}dy \\
  &= s_0 + \bigg\{\log\bigg(\frac12 + s + \sqrt{s+s^2}\bigg) 
	- \log\bigg(\frac12 + s_0 + \sqrt{s_0+s_0^2}\bigg)\bigg\}{\mathbf 1}_{\{s>s_0\}},
\end{align*}
where ${\mathbf 1}_{\{s>s_0\}}$ is the characteristic function on $\{s>s_0\}$.
Note that the entropy $h_0$ and its associated inequality would be difficult to find without
the help of the fast-reaction limit. 
\qed
\end{remark}


\begin{appendix}
\section{Auxiliary results}\label{ref.app}

First, we give an example of functions which satisfy Assumptions (A1)-(A5).

\begin{lemma}[Assumptions (A1)-(A5)]\label{lem.fct}
The functions 
\begin{align*}
  & f_i(u_i) = \alpha_{i}u_i + u_i^{\delta}, \quad
	q_i(u_i) = u_i^\beta, \quad i=1,2,3, \\
	& f_{12}(u_1,u_2) = \alpha u_1^\gamma u_2, \quad 
	f_{21}(u_1,u_2) = \alpha u_1u_2^\gamma,
\end{align*}
satisfy Assumptions (A1)-(A5) if
$$
  \beta\ge 1, \quad \gamma\ge 1, \quad \delta\ge 1+4\max\{\beta,\gamma-1\},
	\quad 1024\alpha^2\le \min\{\alpha_1,\alpha_2,\delta\}.
$$
\end{lemma}

\begin{proof}
Assumptions (A1)-(A3) are satisfied since $\beta$, $\gamma$, $\delta\ge 1$,
Moreover, Assumption (A4) holds if $\delta>2\beta-1$. It remains to verify
Assumption (A5). Multiply the corresponding inequality by 
$q_1(u_1)q_2(u_2)(1+\eta q_1(u_1))(1+\eta q_2(u_2))/(q_1'(u_1)q_2'(u_2))$ 
and abbreviate both sides by
\begin{align*}
   L &:= \frac{q_1(u_1)q_2(u_2)}{q_1'(u_1)q_2'(u_2)}
	(1+\eta q_1(u_1))(1+\eta q_2(u_2)) \\
	&\phantom{xx}{}\times
	\bigg(\frac{q_1'(u_1)\pa_2f_{12}(u_1,u_2)}{q_1(u_1)(1+\eta q_1(u_1))}
	+ \frac{q_2'(u_2)\pa_1f_{21}(u_1,u_2)}{q_2(u_2)(1+\eta q_2(u_2))}\bigg)^2, \\
	R &:= 2(1-\delta)^2
	\big(f_1'(u_1)+\pa_1 f_{12}(u_1,u_2)\big)\big(f_2'(u_2)+\pa_2 f_{21}(u_1,u_2)\big).
\end{align*}
We have to show that $L\le R$ for some $\delta\in (0,1)$. 
We choose $\delta=1-1/\sqrt{2}$.
First, we estimate the right-hand side $R$:
$$
  R \ge 2(1-\delta)^2(\alpha_1+\delta u_1^{\delta-1})
	(\alpha_2+\delta u_2^{\delta-1}) \\
	\ge \min\{\alpha_1,\alpha_2,\delta\}(1+u_1^{\delta-1})(1+u_2^{\delta-1}).
$$
For the left-hand side $L$, we use $\eta\le 1$ and the elementary inequalities
$s\le 1+s^\beta$, $(1+s^\eps)(1+s^\eta)\le 2(1+s^{\max\{\eps,\eta\}})^2$ 
for $s\ge 0$, $\beta\ge 1$, and $\eps$, $\eta>0$:
\begin{align*}
  L &\le \frac{q_1(u_1)q_2(u_2)}{q_1'(u_1)q_2'(u_2)}
	(1+q_1(u_1))(1+q_2(u_2))
	\bigg(\frac{q_1'(u_1)}{q_1(u_1)}\pa_2f_{12}
	+ \frac{q_2'(u_2)}{q_2(u_2)}\pa_1f_{21}\bigg)^2 \\
	&= \alpha^2 u_1u_2(1+u_1^\beta)(1+u_2^\beta)(u_1^{\gamma-1}+u_2^{\gamma-1})^2 \\
	&\le 2\alpha^2(1+u_1^\beta)^2(1+u_2^\beta)^2(u_1^{2(\gamma-1)}+u_2^{2(\gamma-1)}) \\
	&\le 8\alpha^2(1+u_1^{2\beta})(1+u_2^{2\beta})
	\big((1+u_1^{2(\gamma-1)})+(1+u_2^{2(\gamma-1)})\big) \\ 
	&\le 16\alpha^2(1+u_1^{2\max\{\beta,\gamma-1\}})^2(1+u_2^{2\beta})
	+ 16\alpha^2(1+u_1^{2\beta})(1+u_2^{2\max\{\beta,\gamma-1\}})^2 \\
	&\le 64\alpha^2(1+u_1^{2\max\{\beta,\gamma-1\}})^2
	(1+u_2^{2\max\{\beta,\gamma-1\}})^2 \\
	&\le 256\alpha^2(1+u_1^{4\max\{\beta,\gamma-1\}})
	(1+u_2^{4\max\{\beta,\gamma-1\}}) \\
	&\le 1024\alpha^2(1+u_1^{\delta-1})(1+u_2^{\delta-1}),
\end{align*}
if $4\max\{\beta,\gamma-1\}\le\delta-1$.
Then $L\le R$ if additionally $1024\alpha^2\le\min\{\alpha_1,\alpha_2,\delta\}$. 
Note that these conditions are far from being optimal.
\end{proof}

We then turn to the

\begin{lemma}[$F$ is a homeomorphism]\label{lem.hom}
The function $F=(F_1,F_2,F_3):\R_+^3\to\R_+^3$ defined in \eqref{1.F}
is a $C^1$-diffeomorphism from $(\R_+^*)^3$ into itself and a homeomorphism
from $\R_+^3$ into itself.
\end{lemma}

\begin{proof}
We follow the strategy of \cite[Section 4.2]{LeMo16}. The proof of \cite{LeMo16}
is valid for functions $F_i(s)/s_i$ whose variables separate. Since this is not the
case in our situation, we need to modify the proof.

{\em Step 1: $F$ is a $C^1$-diffeomorphism on $(\R_+^*)^3$.}
Introduce the function $\Phi=(\Phi_1,\Phi_2,\Phi_3):\R^3\to\R^3$, 
$\Phi_i(s)=\log(F_i(\exp s))$, where log and exp are defined coordinate-wise.
By the Hadamard-L\'evy theorem, a $C^1$ self-mapping is a $C^1$-diffeomorphism
if and only if it is proper and has no critical points. The proof that $\Phi$ is
proper is exactly as in \cite[Section 4.2]{LeMo16}, since here, the separability
of variables is not needed. To show that $\Phi$ has no critical points, we compute
the determinant
\begin{align*}
  \det F'(s) &= f_3'(s_3)\Big(\big(f_1'(s_1)+\pa_1 f_{12}(s_1,s_2)\big)
	\big(f_2'(s_2)+\pa_2 f_{21}(s_1,s_2)\big) \\
	&\phantom{xx}{}- \pa_2 f_{12}(s_1,s_2)\pa_1 f_{21}(s_1,s_2)\Big) > 0,
\end{align*}
which is positive
because of Assumption (A5) (see Remark \ref{rem.ass}). Since both 
$\log:(\R_+^*)^3\to\R^3$ and $\exp:\R^3\to(\R_+^*)^3$ have no critical points,
we conclude that also $\Phi$ has no critical points. By the Hadamard-L\'evy
theorem, $\Phi:\R^3\to\R^3$ is a $C^1$-diffeomorphism and so does
$F:(\R_+^*)^3\to(\R_+^*)^3$.

{\em Step 2: $F$ is bijective on $\R_+^3$.}
It remains to treat the boundary of $\R_+^3$. To this end, we split it as
$\pa\R_+^3=\{0,0,0\}\cup V$, where $V=V_1\cup\cdots\cup V_6$ and the sets 
$V_i$ are either a quarter-plane or a half-line.
Since $F_i$ can be written as the product of $s_i$ and some
nonnegative function, we have $F(V_i)\subset V_i$. We show that $F$ is bijective
on each $V_i$. 

As the six cases are similar, we give only a proof for 
$V_3=\R_+^*\times\R_+^*\times\{0\}$. The result follows when we have
shown that $F\in C^1(V_3;V_3)$ is proper and has no critical points. 
Let $s=(s_1,s_2,0)\in V_3$ and 
define the induced vector $\bar s=(s_1,s_2)\in(\R_+^*)^2$
and the induced function $\bar F:(\R_+^*)^2\to(\R_+^*)^2$,
$\bar F_i(\bar s)=F_i(s_1,s_2,0)$, $i=1,2$. 
Clearly, $\bar F$ is proper on $(\R_+^*)^2$ since $F$ is proper on $(\R_+^*)^3$.
We prove that $\bar F$ is $C^1$ and has no critical points. The first property
is clear and the second one follows from
$$
  \det \bar F'(\bar s) = \big(f_1'(\bar s_1)+\pa_1 f_{12}(\bar s)\big)
	\big(f_2'(\bar s_2)+\pa_2 f_{21}(\bar s)\big) 
	- \pa_2 f_{12}(\bar s)\pa_1 f_{21}(\bar s) > 0.
$$
By the Hadamard-L\'evy theorem, $\bar F$ is a $C^1$-diffeomorphism on $(\R_+^*)^2$.
By construction of $\bar F$, this implies that $F$ is bijective on $V_3$.

{\em Step 3: $F^{-1}$ is continuous on $\R_+^3$.}
Let $(y_n)\subset(\R_+^*)^3$ be such that $y_n\to y$ as $n\to\infty$.
If $y\in(\R_+^*)^3$, then we already know that $F^{-1}(y_n)\to F^{-1}(y)$.
Thus, let $y\in V$. Since $F$ is proper, $(F^{-1}(y_n))$ is bounded
and there exists a subsequence (not relabeled) such that $F^{-1}(y_n)\to\widetilde y$
for some $\widetilde y$. As $F$ is one-to-one and continuous, we obtain
$F(\widetilde y)=y$. Consequently, $F^{-1}(y_n)\to F^{-1}(y)$. We infer that
$F^{-1}$ is continuous on $\R_+^3$, which concludes the proof.
\end{proof}

The following result, used in Lemma \ref{lem.L2}, is proved in \cite[Lemma 2.12]{DLMT15}.

\begin{lemma}[Discrete duality estimate]\label{lem.dual}
Let $N\in\N$, $\rho>0$, and $\tau>0$ be such that $\rho\tau<1$ and set $T:=N\tau$.
Let $\mu^1,\ldots,u^N$ be nonnegative integrable functions and let
$u^0,\ldots,u^N$ be nonnegative bounded functions satisfying
$\mu^k u^k\in H^2(\Omega)$ and, for $1\le k\le N$,
$$
  \frac{1}{\tau}(u^k-u^{k-1}) - \Delta(\mu^k u^k) \le \rho u^k\quad\mbox{in }\Omega,
	\quad \na(\mu^k u^k)\cdot\nu=0\quad\mbox{on }\pa\Omega.
$$
Then there exists a constant $C>0$ only depending on $\|u^0\|_{L^2(\Omega)}$,
$\Omega$, $\rho$, and $T$, such that
$$
  \tau\sum_{k=1}^N\int_\Omega\mu^k(u^k)^2 dx 
	\le C\bigg(1 + \tau\sum_{k=1}^N\int_\Omega \mu^k dx\bigg).
$$
\end{lemma}

Finally, we recall two useful regularity results for elliptic equations.

\begin{lemma}[Theorem 2.3.3.6 in \cite{Gri85}]\label{lem.regul1}
Let $p\in(1,\infty)$ and let $\Omega\subset\R^d$ be a bounded domain with
smooth boundary. Then there exist positive constants $K_{p,\Omega}$ and
$C_{p,\Omega}$ such that for all $M>K_{p,\Omega}$ and all $u\in W^{2,1}(\Omega)$
satisfying
$$
  Mu-\Delta u = g\in L^p(\Omega), \quad \na u\cdot\nu=0\quad\mbox{on }\pa\Omega,
$$
it holds that
$$
  \|u\|_{W^{2,p}(\Omega)} \le C_{p,\Omega}\|g\|_{L^p(\Omega)}.
$$
\end{lemma}

\begin{lemma}[Lemma 6.6 in \cite{DLMT15}]\label{lem.regul2}
Let $\Omega\subset\R^d$ be a bounded domain, $f\in L^p(\Omega)$ with $p>d/2$, 
and let $u\in H^2(\Omega)$ satisfy $u\ge 0$ in $\Omega$ and
$$
  -\Delta u \le f\quad\mbox{in }\Omega, \quad \na u\cdot\nu=0\quad\mbox{on }\pa\Omega.
$$
Then there exists a positive constant only depending on $\Omega$ such that
$$
  \|u\|_{L^\infty(\Omega)} \le C\big(\|f\|_{L^p(\Omega)} + \|u\|_{L^1(\Omega)}\big).
$$
\end{lemma}

\end{appendix}


\end{document}